\documentclass[a4paper]{ijbc_la}
\usepackage{amsmath,amssymb}
\usepackage{amsfonts}
\usepackage{enumerate}
\usepackage{epsfig,afterpage}
\usepackage[dvips]{psfrag}
\usepackage{psfrag}
\usepackage{graphicx}
\usepackage[active]{srcltx}
\usepackage{multirow}

\def\be{\begin{equation}}
\def\ee{\end{equation}}
\def\ba{\begin{array}}
\def\ea{\end{array}}
\def\bq{\begin{eqnarray}}
\def\eq{\end{eqnarray}}
\def\beq{\begin{eqnarray*}}
\def\eeq{\end{eqnarray*}}
\def\bi{\begin{itemize}}
\def\ei{\end{itemize}}
\def\bc{\begin{center}}
\def\ec{\end{center}}

\def\bdf{\begin{definition}}
\def\edf{\end{definition}}
\def\bal{\begin{aligned}}
\def\eal{\end{aligned}}
\def\bth{\begin{theorem}}
\def\eth{\end{theorem}}
\def\brm{\begin{remark}}
\def\erm{\end{remark}}
\def\la{\lambda}

\def\o{\omega}

\def\p{\partial}
\def\deg{\mbox{deg}}

\def\Res{\mbox{Res}}

\def\ov{\overline}
\def\Res{\operatorname{Res}}

\newcommand{\R}{\mathbb{R}}
\newcommand{\C}{\mathbb{C}}
\DeclareMathOperator{\SC}{\mathbb{S}}

\newcommand{\N}{\mathbb{N}}

\newcommand{\PP}{\mathbb{P}}

\title{Global phase portraits of quadratic polynomial differential systems with a semi--elemental triple node} 

\author{Joan C. Art\'es}
\address{Departament de Matem\`{a}tiques, Universitat Aut\`{o}noma de
Barcelona, 08193 Bellaterra, Barcelona, Spain\\E--mail: artes@mat.uab.cat}

\author{Alex C. Rezende$^1$ and Regilene D. S. Oliveira$^2$}
\address{Departamento de Matem\'{a}tica,
Universidade de S\~{a}o Paulo, \\13566--590, S\~{a}o Carlos,
S\~{a}o Paulo, Brazil, \\E--mail:
$^1$arezende@icmc.usp.br,
$^2$regilene@icmc.usp.br}

\date{}

\abstract{
Planar quadratic differential systems occur in many areas of applied mathematics. Although more than one thousand papers have been
written on these systems, a complete understanding of this family is still missing. Classical problems, and in particular, Hilbert's 16th problem \cite{Hilbert:1900,Hilbert:1902}, are still open for this family. In this article we make a global study of the family $QT\overline{N}$ of all real quadratic polynomial differential systems which have a semi--elemental triple node (triple node with exactly one zero eigenvalue). This family modulo the action of the affine group and time homotheties is three--dimensional and we give its bifurcation diagram with respect to a normal form, in the three-dimensional real space of the parameters of this form. This bifurcation diagram yields 28 phase portraits for systems in $QT\overline{N}$ counting phase portraits with and without limit cycles. Algebraic invariants are used to construct the bifurcation set. The phase portraits are represented on the Poincar\'{e} disk. The bifurcation set is not only algebraic due to the presence of a surface found numerically. All points in this surface correspond to connections of separatrices.}

\runningheads{J.C. Art\'{e}s, A.C. Rezende and R.D.S. Oliveira}{Global phase portraits of quadratic polynomial differential systems with a semi--elemental triple node}

\newpage

\begin{document}

\maketitle\clearpage

\section{Introduction, brief review of the literature and statement of results}\label{sec:int}

\indent In this paper we call \textit{quadratic differential systems} or simply \textit{quadratic systems}, differential systems of the form
    \be \ba{lcccl}
            \dot{x}&=& p(x,y), \\
            \dot{y}&=& q(x,y), \\
    \ea \label{eq:qs} \ee
where $p$ and $q$ are polynomials over $\R$ in $x$ and $y$ such that the max(\deg$(p)$,\deg$(q))=2$. To such a system one can always associate the quadratic vector field
    \be
        X=p\frac{\p}{\p x}+q\frac{\p}{\p y} \label{eq:qvf},
    \ee
as well as the differential equation
    \be
        qdx-pdy=0.
        \label{eq:de}
    \ee
The class of all quadratic differential systems (or quadratic vector fields) will be denoted by $QS$.

We can also write system \eqref{eq:qs} as
    \be \ba{lcccl}
        \dot{x} & = p_0+p_{1}(x,y)+p_{2}(x,y)=p(x,y), \\
        \dot{y} & = q_0+q_{1}(x,y)+q_{2}(x,y)=q(x,y), \\
    \ea \label{2l1} \ee
where $p_i$ and $q_i$ are homogeneous polynomials of degree $i$ in $(x,y)$ with real coefficients with $p_{2}^2+q_{2}^2 \neq 0$.

The complete characterization of the phase portraits for real planar quadratic vector fields is not known and attempting to topologically classify these systems, which occur rather often in applications, is quite a complex task. This family of systems depends on twelve parameters, but due to the action of the group $G$ of real affine transformations and time homotheties, the class ultimately depends on five parameters, but this is still a large number.

The goal of this article is to study the class $QT\overline{N}$ of all quadratic systems possessing a semi--elemental triple node. By a semi--elemental point we understand a singular point with zero determinant of its Jacobian, but only one eigenvalue zero. These points are known in classical literature as semi--elementary, but we use the term semi--elemental introduced in \cite{Artes-Llibre-Schlomiuk-Vulpe:2012} as part of more consistent definitions associated to singular points, their multiplicities and Jacobians.

The condition of having a semi--elemental triple node of all the systems in $QT\overline{N}$ implies that these systems may have another finite point or not.

For a general framework of study of the class of all quadratic differential systems we refer to the article of Roussarie and Schlomiuk \cite{Roussarie-Schlomiuk:2002}.

In this study we follow the pattern set out in \cite{Artes-Llibre-Schlomiuk:2006}. As much as possible we shall try to avoid repeating technical sections which are the same for both papers, referring to the paper mentioned just above, for more complete information.

In this article we give a partition of the class $QT\overline{N}$ into 63 parts: 17 three--dimensional ones, 29 two--dimensional ones, 15 one--dimensional ones and 2 points. This partition is obtained by considering all the bifurcation surfaces of singularities and one related to connections of separatrices, modulo ``islands'' (see Sec. \ref{sec:islands}).

A \textit{graphic} as defined in \cite{Dumortier-Roussarie-Rousseau:1994} is formed by a finite sequence of points $r_1,r_2,\ldots,r_n$ (with possible repetitions) and non--trivial connecting orbits $\gamma_i$ for $i=1,\ldots,n$ such that  $\gamma_i$ has $r_i$ as $\alpha$--limit set and $r_{i+1}$ as $\omega$--limit set for $i<n$ and $\gamma_n$ has $r_n$ as  $\alpha$--limit set and $r_{1}$ as $\omega$--limit set. Also normal orientations $n_j$ of the non--trivial orbits must be coherent  in the sense that if $\gamma_{j-1}$ has left--hand orientation then so does $\gamma_j$. A \textit{polycycle} is a graphic which has  a Poincar\'{e} return map. For more details, see \cite{Dumortier-Roussarie-Rousseau:1994}.

\begin{theorem} \label{th:1.1} There exist 28 distinct phase portraits for the quadratic vector fields having a semi--elemental triple node. All these phase portraits are shown in Fig. \ref{fig:phase1}. Moreover, the following statements hold:
\begin{enumerate}[(a)]
\item There exist three phase portraits with limit cycles, and they are in the regions $V_{6}$, $V_{15}$ and $5S_{5}$;
\item There exist four phase portraits with graphics, and they are in the regions $5S_{4}$, $7S_{1}$, $1.3L_{2}$ and $5.7L_{1}$.
\end{enumerate}
\end{theorem}

\onecolumn
\begin{figure}
\psfrag{V1}{$V_{1}$}   \psfrag{V3}{$V_{3}$}   \psfrag{V4}{$V_{4}$}
\psfrag{V6}{$V_{6}$}   \psfrag{V8}{$V_{8}$}   \psfrag{V10}{$V_{10}$}
\psfrag{XX}{$V_{11}$} \psfrag{V12}{$V_{12}$} \psfrag{V15}{$V_{15}$}
\psfrag{1S1}{$1S_{1}$} \psfrag{1S2}{$1S_{2}$} \psfrag{1S3}{$1S_{3}$}
\psfrag{1S4}{$1S_{4}$} \psfrag{5S1}{$5S_{1}$} \psfrag{5S2}{$5S_{2}$}
\psfrag{5S4}{$5S_{4}$} \psfrag{5S5}{$5S_{5}$} \psfrag{5S7}{$5S_{7}$}
\psfrag{5S8}{$5S_{8}$} \psfrag{5S9}{$5S_{9}$} \psfrag{5S10}{$5S_{10}$}
\psfrag{7S1}{$7S_{1}$} \psfrag{1.3L1}{$1.3L_{1}$} \psfrag{1.3L2}{$1.3L_{2}$}
\psfrag{1.5L1}{$1.5L_{1}$} \psfrag{1.5L2}{$1.5L_{2}$} \psfrag{5.7L1}{$5.7L_{1}$}
\psfrag{P1}{$P_{1}$}
\centerline{\psfig{figure=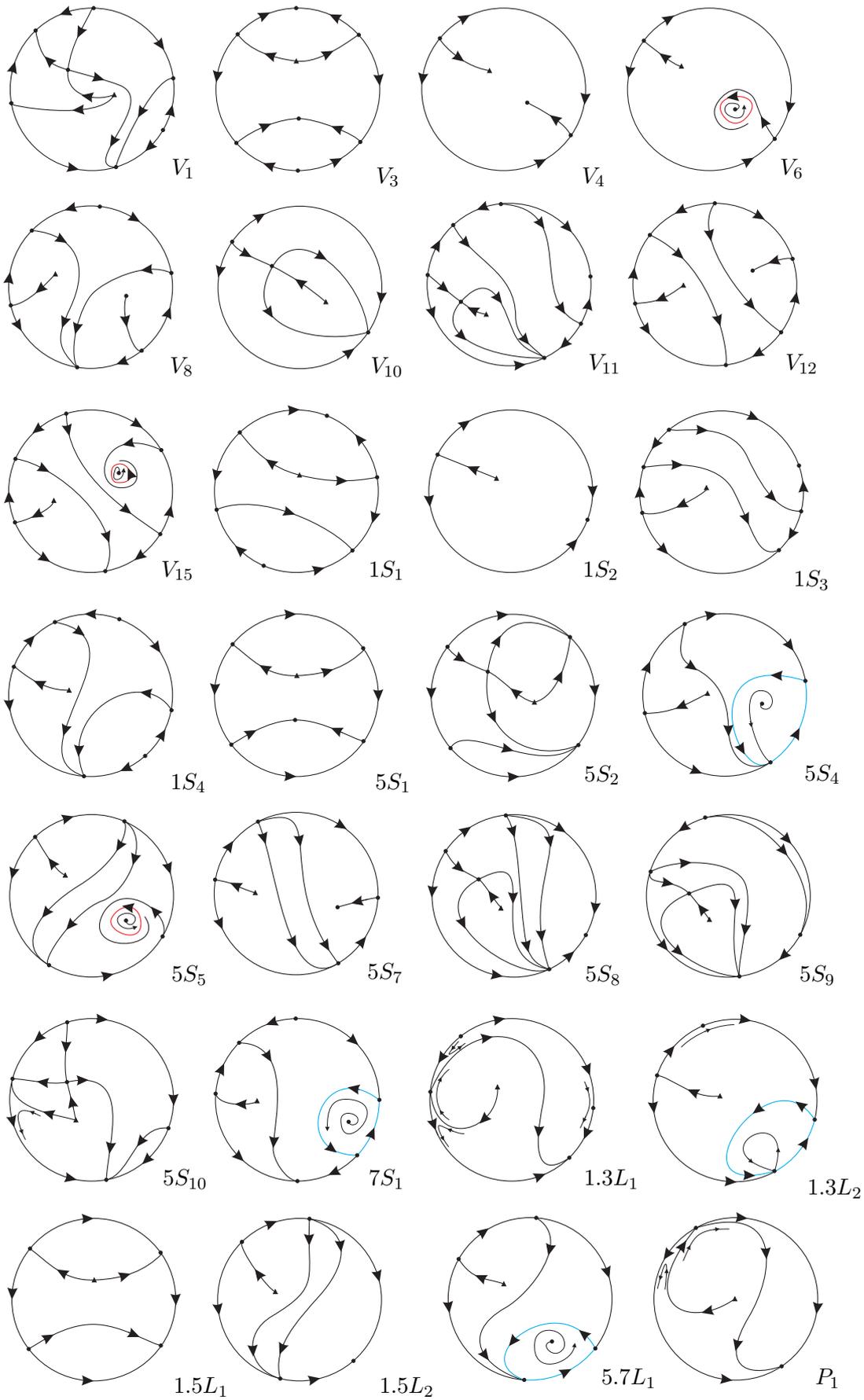,width=14cm}} \centerline {}
\caption{\small \label{fig:phase1} Phase portraits for quadratic vector fields with a semi--elemental triple node.}
\end{figure}
\twocolumn

From the 28 different phase portraits, 9 occur in 3--dimensional parts, 13 in 2--dimensional parts, 5 in 1--dimensional parts and 1 occur in a single 0--dimensional part.

In Fig. \ref{fig:phase1} we have denoted with a little disk the elemental singular points and with a little triangle the semi--elemental triple node. We have plotted with wide curves the separatrices and we have added some thinner orbits to avoid confusion in some required cases.

\begin{remark} \rm The phase portraits are labeled according to the parts of the bifurcation diagram where they occur. These labels could be different for two topologically equivalent phase portraits occurring in distinct parts. Some of the phase portraits in 3--dimensional parts also occur in some 2--dimensional parts bordering these 3--dimensional parts. An example occurs when a node turns into a focus. An analogous situation happens for phase portraits in 2--dimensional (respectively, 1--dimensional) parts, coinciding with a phase portrait on 1--dimensional (respectively, 0--dimensional) part situated on the border of it.
\end{remark}

The work is organized as follows. In Sec. \ref{sec:qvftn} we describe the normal form for the family of systems having a semi--elemental triple node.

For the study of real planar polynomial vector fields two compactifications are used. In Sec. \ref{sec:poincare} we describe very briefly the Poincar\'{e} compactification on the 2--dimensional sphere.

In Sec. \ref{sec:basicwf} we list some very basic properties of general quadratic systems needed in this study.

In Sec. \ref{sec:internum} we mention some algebraic and geometric concepts that were introduced in \cite{Schlomiuk-Pal:2001,Llibre-Schlomiuk:2004} involving intersection numbers, zero--cycles, divisors, and T--comitants and invariants for quadratic systems as used by the Sibirskii school. We refer the reader directly to \cite{Artes-Llibre-Schlomiuk:2006} where these concepts are widely explained.

In Sec. \ref{sec:bifur}, using algebraic invariants and T--comitants, we construct the bifurcation surfaces for the class $QT\overline{N}$.

In Sec. \ref{sec:8} we introduce a global invariant denoted by ${\cal I}$, which classifies completely, up to topological equivalence, the phase portraits we have obtained for the systems in the class $QT\overline{N}$. Theorem \ref{t12} shows clearly that they are uniquely determined (up to topological equivalence) by the values of the invariant ${\cal I}$.

\section{Quadratic vector fields with a semi--elemental triple node}\label{sec:qvftn}

\indent A singular point $r$ of a planar vector field $X$ in $\R^2$ is \textit{semi--elemental} if the determinant of the matrix of its linear part, $DX(r)$, is zero, but its trace is different from zero.

The following result characterizes the local phase portrait at a semi--elemental singular point.

\begin{proposition} \label{th2.19} \cite{Andronov-Leontovich-Gordon-Maier:1973,Dumortier-Llibre-Artes:2006}
Let $r=(0,0)$ be an isolated singular point of the vector field $X$ given by
    \begin{equation}
        \begin{array}{ccl}
            \dot{x} & = & A(x,y), \\
            \dot{y} & = & \lambda y + B(x,y), \\
        \end{array}
        \label{eqth2.19}
    \end{equation}
where $A$ and $B$ are analytic in a neighborhood of the origin starting with at least degree 2 in the variables $x$ and $y$. Let $y=f(x)$ be the solution of the equation $\lambda y + B(x,y) = 0$ in a neighborhood of the point $r=(0,0)$, and suppose that the function $g(x) = A(x,f(x))$ has the expression $g(x) = a x^\alpha + o(x^\alpha)$, where $\alpha \geq 2$ and $a \neq 0$. So, when $\alpha$ is odd, then $r=(0,0)$ is either an unstable multiple node, or a multiple saddle, depending if $a>0$, or $a<0$, respectively. In the case of the multiple saddle, the separatrices are tangent to the $x$--axis. If $\alpha$ is even, the $r=(0,0)$ is a multiple saddle--node, i.e. the singular point is formed by the union of two hyperbolic sectors with one parabolic sector. The stable separatrix is tangent to the positive (respectively, negative) $x$--axis at $r=(0,0)$ according to $a<0$ (respectively, $a>0$). The two unstable separatrices are tangent to the $y$--axis at $r=(0,0)$.
\end{proposition}

In the particular case where $A$ and $B$ are real quadratic polynomials in the variables $x$ and $y$, a quadratic system with a semi--elemental singular point at the origin can always be written into the form
    \begin{equation}
        \begin{array}{ccl}
            \dot{x} & = & g x^2 + 2h xy + k y^2, \\
            \dot{y} & = & y + \ell x^2 + 2m xy + n y^2. \\
        \end{array}
        \label{eqsemielemental}
    \end{equation}

By Proposition \ref{th2.19}, if $g \neq 0$, then we have a double saddle--node $\overline{sn}_{(2)}$, using the notation introduced in \cite{Artes-Llibre-Schlomiuk-Vulpe:2012}. Otherwise, if $g = 0$ and $\ell \neq 0$, then, if $\ell < 0$, we have a triple node $\bar{n}_{(3)}$ and, if $\ell > 0$, we have a triple saddle $\bar{s}_{(3)}$. The multiplicity of the singular point is 3 (if $\ell \neq 0$) as can be checked using the invariants and tables given in \cite{Artes-Llibre-Vulpe:2008}.

In the normal form above, we consider the coefficient of the terms $xy$ in both equations multiplied by 2 in order to make easier the calculations of the algebraic invariants we shall compute later.

After an affine transformation and a time homothety, a quadratic differential system with a semi--elemental triple node can be put into the following normal form:
    \begin{equation}
        \begin{array}{ccl}
            \dot{x} & = & 2 xy + k y^2, \\
            \dot{y} & = & y - x^2 + 2m xy + n y^2, \\
        \end{array}
        \label{eqtn}
    \end{equation}
where $m$, $n$ and $k$ are real parameters. We will study the bifurcations of system \eqref{eqtn} in the 3--dimensional space with coordinates $(m,n,k) \in \R^3$.

\begin{remark} \rm
After rescaling the time, we note that system \eqref{eqtn} is symmetric in relation to the real parameter $k$. So that, we will only consider $k \geq 0$.
\end{remark}

\section{The Poincar\'{e} compactification and the complex (real) foliation with singularities on $\C\PP^2$ ($\R\PP^2$)} \label{sec:poincare}

\indent A real planar polynomial vector field $\xi$ can be compactified on the sphere as follows. Consider the $x$, $y$ plane as being the plane $Z=1$ in the space $\R^3$ with coordinates $X$, $Y$, $Z$. The central projection of the vector field $\xi$ on the sphere of radius one yields a diffeomorphic vector field on the upper hemisphere and also another vector field on the lower hemisphere. There exists (for a proof see \cite {Gonzales:1969}) an analytic vector field $cp(\xi)$ on the whole sphere such that its restriction on the upper hemisphere has the same phase curves as the one constructed above from the polynomial vector field. The projection of the closed northern hemisphere $H^+$ of $\,\SC^2$ on $Z=0$ under $(X,Y,Z) \to (X,Y)$ is called \textit{the Poincar\'{e} disc}. A singular point $q$ of $cp(\xi)$ is called an \textit{infinite} (respectively, \textit{finite}) singular point if $q\in \SC^1$, the equator (respectively, $q\in \SC^2 \setminus \SC^1$). By the \textit{Poincar\'{e} compactification of a polynomial vector field} we mean the vector field $cp(\xi)$ restricted to the upper hemisphere completed with the equator.

Ideas in the remaining part of this section go back to Darboux's work \cite{Darboux:1878}. Let $p(x,y)$ and $q(x,y)$ be polynomials with real coefficients. For the vector field
    \begin{equation}\label{21}
        p\frac{\p}{\p x}+ q\frac{\p}{\p y},
    \end{equation}
or equivalently for the differential system
    \begin{equation}\label{22}
        \dot x = p(x,y), \qquad \dot y= q(x,y),
    \end{equation}
we consider the associated differential $1$--form \linebreak $\o_{1}= q(x,y)dx- p(x,y)dy$, and the differential equation
    \begin{equation}\label{23}
        \o_{1}=0 \ .
    \end{equation}
Clearly, equation \eqref{23} defines a foliation with singularities on $\C^2$. The affine plane $\C^2$ is compactified on the complex projective space $\C\PP^2= (\C^3\setminus \{0\})/\sim$, where $(X,Y,Z)\sim (X',Y',Z')$ if and only if $(X,Y,Z)= \la (X',Y',Z')$ for some complex $\la\ne 0$. The equivalence class of $(X,Y,Z)$ will be denoted by $[X:Y:Z]$.

The foliation with singularities defined by equation \eqref{23} on $\C^2$ can be extended to a foliation with singularities on $\C\PP^2$ and the $1$--form $\o_{1}$ can be extended to a meromorphic $1$--form $\o$ on $\C\PP^2$ which yields an equation $\o=0$, i.e.
    \begin{equation}\label{ja1}
        A(X,Y,Z)dX+B(X,Y,Z)dY+C(X,Y,Z)dZ=0,
    \end{equation}
whose coefficients $A$, $B$, $C$ are homogeneous polynomials of the same degree and satisfy the relation:
    \begin{equation}\label{ja2}
        A(X,Y,Z)X+B(X,Y,Z)Y+C(X,Y,Z)Z=0,
    \end{equation}
Indeed, consider the map $i: \C^3 \setminus \{Z = 0\} \to \C^2$, given by $i(X,Y,Z)= (X/Z,Y/Z)=(x,y)$ and suppose that $\max\{\deg(p),\deg(q)\}= m>0$. Since $x=X/Z$ and $y=Y/Z$ we have:
    \[
        dx= (ZdX-XdZ)/Z^2, \quad dy= (ZdY-YdZ)/Z^2,
    \]
the pull--back form $i^*(\o_{1})$ has poles at $Z=0$ and yields the equation
    \[\aligned
        i^*(\o_{1})=& q(X/Z,Y/Z) (ZdX-XdZ)/Z^2 \\
                    & -p(X/Z,Y/Z) (ZdY-YdZ)/Z^2 = 0.
    \endaligned \]
Then, the $1$--form $\o= Z^{m+2} i^*(\o_{1})$ in $\C^3\setminus \{Z\ne 0\}$ has homogeneous polynomial coefficients of degree $m+1$, and for $Z=0$ the equations $\o=0$ and $i^*(\o_{1})=0$ have the same solutions. Therefore, the differential equation $\o=0$ can be written as \eqref{ja1}, where
    \be\aligned\label{24}
        A(X,Y,Z) =& Z Q(X,Y,Z)= Z^{m+1} q(X/Z,Y/Z),\\
        B(X,Y,Z) =& -Z P(X,Y,Z)=-Z^{m+1} p(X/Z,Y/Z),\\
        C(X,Y,Z) =& Y P(X,Y,Z)- X Q(X,Y,Z).
    \endaligned\ee

Clearly $A$, $B$ and $C$ are homogeneous polynomials of degree $m+1$ satisfying \eqref{ja2}.

In particular, for our quadratic systems \eqref{eqtn}, $A$, $B$ and $C$ take the following forms
    $$\begin{aligned}
        A(X,Y,Z)=& ZQ(X,Y,Z)=-Z(X^2-2mXY\\
                 &-nY^2-YZ),\\
        B(X,Y,Z)=& -ZP(X,Y,Z)=-YZ(2X+kY),
    \end{aligned}$$
    \begin{equation} \label{ja3} \begin{aligned}
        C(X,Y,Z)=& YP(X,Y,Z)-XQ(X,Y,Z) \\
                =& X^3-2mX^2Y+2XY^2\\
                 &-nXY^2+kY^3-XYZ.
    \end{aligned} \end{equation}

We note that the straight line $Z=0$ is always an algebraic invariant curve of this foliation and that its singular points are the solutions of the system: $A(X,Y,Z)= B(X,Y,Z)= C(X,Y,Z)=0$. We note also that $C(X,Y,Z)$ does not depend on $b$.

To study the foliation with singularities defined by the differential equation \eqref{ja1} subject to \eqref{ja2} with $A$, $B$, $C$ satisfying the above conditions in the neighborhood of the line $Z=0$, we consider the two charts of $\C\PP^2$: $(u,z)= (Y/X,Z/X)$, $X\ne 0$, and $(v,w)= (X/Y,Z/Y)$, $Y\ne 0$, covering this line. We note that in the intersection of the charts $(x,y)= (X/Z,Y/Z)$ and $(u,z)$ (respectively, $(v,w)$) we have the change of coordinates $x=1/z$, $y=u/z$ (respectively, $x=v/w$, $y=1/w$). Except for the point $[0:1:0]$ or the point $[1:0:0]$, the foliation defined by equations \eqref{ja1},\eqref{ja2} with $A$, $B$, $C$ as in \eqref{24} yields in the neighborhood of the line $Z=0$ the foliations associated with the systems
    \be \begin{aligned}
        \dot u =& uP(1,u,z)-Q(1,u,z)= C(1,u,z), \\
        \dot z=& z P(1,u,z), \label{26}
    \end{aligned} \ee
or
    \be \begin{aligned}
        \dot v =& vQ(v,1,w)-P(v,1,w)= -C(v,1,w), \\
        \dot w =& w P(v,1,w). \label{27}
    \end{aligned} \ee

In a similar way we can associate a real foliation with singularities on $\R\PP^2$ to a real planar polynomial vector field.

\section{A few basic properties of quadratic systems relevant for this study} \label{sec:basicwf}

\indent We list below results which play a role in the study of the global phase portraits of the real planar quadratic systems \eqref{eq:qs} having a semi--elemental triple node.

The following results hold for any quadratic system:
    \begin{enumerate}[(i)]
        \item \label{item_i} A straight line either has at most two (finite) contact points with a quadratic system (which include the singular points), or it is formed by trajectories of the system; see Lemma 11.1 of \cite{Ye:1986}. We recall that by definition a {\it contact point} of a straight line $L$ is a point of $L$ where the vector field has the same direction as $L$, or it is zero.

        \item \label{item_ii} If a straight line passing through two real finite singular points $r_{1}$ and $r_{2}$ of a quadratic system is not formed by trajectories, then it is divided by these two singular points in three segments $\ov{\infty r_{1}}$, $\ov{r_{1} r_{2}}$ and $\ov{r_{2} \infty}$ such that the trajectories cross $\ov{\infty r_{1}}$ and $\ov{r_{2}\infty}$ in one direction, and they cross $\ov{r_{1} r_{2}}$ in the opposite direction; see Lemma 11.4 of \cite{Ye:1986}.

        \item If a quadratic system has a limit cycle, then it surrounds a unique singular point, and this point is a focus; see \cite{Coppel:1966}.
    \end{enumerate}

\begin{theorem}\label{th:4.2} Any graphic or degenerated graphic in a real planar polynomial differential system must either
\begin{enumerate}[1)]
\item surround a singular point of index greater than or equal to +1, or
\item contain a singular point having an elliptic sector situated in the region delimited by the graphic, or
\item contain an infinite number of singular points.
\end{enumerate}
\end{theorem}

\begin{proof} Let $S$ be a simply connected invariant closed set under the flow of a vector field. In \cite{Artes-Kooij-Llibre:1998} the \textit{index of $\partial S$} is given by: $\sum^n_{i=1} (E_i-H_i+1)/2$, where $E_i$ (respectively, $H_i$) is the number of elliptic (respectively, hyperbolic) sectors which are inside the region delimited by $S$ of the singular points forming the border. Also the \textit{index of $S$} is given by the index of $\partial S$ plus the sum of the indices of the singular points in the interior of $S$.

From the same paper, Proposition 4.8 claims that given a vector field $X$ or $p(X)$ and $S$ an invariant region topologically equivalent to $\mathbb D^2$ (the closed disk) containing a finite number of singular points (both in $\partial S$ or its interior), then the index of $S$ is always $+1$.

Now, assume that we have a graphic of a polynomial system. If it contains an infinite number of singular points (either finite or infinite) we are done. Otherwise, such a graphic altogether with its interior is an invariant region as defined in \cite{Artes-Kooij-Llibre:1998} and must have index +1. Since the index is positive, we must have some element, either in the interior or on the border which makes the index positive, and this implies the existence of either a point of index greater than or equal to +1 in its interior or at least one elliptic sector coming from a singular point on the border and situated in the region delimited by the graphic.
\end{proof}

\section{Some algebraic and geometric concepts} \label{sec:internum}

\indent In this article we use the concept of intersection number for curves (see \cite{Fulton:1969}). For a quick summary see Sec. 5 of \cite{Artes-Llibre-Schlomiuk:2006}.

We shall also use the concepts of zero--cycle and divisor (see \cite{Hartshorne:1977}) as specified for quadratic vector fields in \cite{Schlomiuk-Pal:2001}. For a quick summary see Sec. 6 of \cite{Artes-Llibre-Schlomiuk:2006}.

We shall also use the concepts of algebraic invariant and T--comitant as used by the Sibirskii school for differential equations. For a quick summary see Sec. 7 of \cite{Artes-Llibre-Schlomiuk:2006}.

The invariants which are relevant in this study are among the ones which were relevant in the study of $QW_{2}$ and will be described in the next section.

\section{The bifurcation diagram of the systems with a semi--elemental triple node} \label{sec:bifur}

\subsection{Bifurcation surfaces due to the changes in the nature of singularities}

For systems \eqref{eqtn} we will always have $(0,0)$ as a finite singular point, a semi--elemental triple node.

From Sec. 7 of \cite{Artes-Llibre-Vulpe:2008} we get the formulas which give the bifurcation surfaces of singularities in $\R^{12}$, produced by changes that may occur in the local nature of finite singularities. From \cite{Schlomiuk-Vulpe:2005} we get equivalent formulas for the infinite singular points. These bifurcation surfaces are all algebraic and they are the following:

\medskip

\noindent \textbf{Bifurcation surfaces in $\R^3$ due to multiplicities of singularities}

\medskip

\noindent {\bf (${\cal S}_{1}$)} This is the bifurcation surface due to multiplicity of infinite singularities as detected by the coefficients of the divisor $D_\R(P,Q;Z)= \sum_{W\in \{Z=0\}\cap \C\PP^2} I_W (P,Q) W$, (here $I_W (P,Q)$ denotes the intersection multiplicity of $P=0$ with $Q=0$ at the point $W$ situated on the line at infinity, i.e. $Z=0$) whenever $deg((D_\R(P,Q;Z)))>0$. This occurs when at least one finite singular point collides with at least one infinite point. More precisely this happens whenever the homogenous polynomials of degree two, $p_2$ and $q_2$ in $p$ and $q$ have a common root. In other words whenever \linebreak $\mu=$ Res$_x(p_2,q_2)/y^4= 0$. The equation of this surface is $$\mu=k^2+4km-4n=0.$$

\medskip

\noindent {\bf (${\cal S}_{5}$)}\footnote{The numbers attached to these bifurcations surfaces do not appear here in increasing order. We just kept the same enumeration used in \cite{Artes-Llibre-Schlomiuk:2006} to maintain coherence even though some of the numbers in that enumeration do not occur here.} This is the bifurcation surface due to multiplicity of infinite singularities as detected by the coefficients of $D_\C(C,Z)=\sum_{W\in \{Z=0\}\cap \C\PP^2} I_W (C,Z) W$, i.e. this bifurcation occurs whenever at a point $W$ of intersection of $C=0$ with $Z=0$ we have $I_W(C,Z)\ge 2$, i.e. when at least two infinite singular points collide at $W$. This occurs whenever the discriminant of $C_2=C(X,Y,0)=Yp_2(X,Y)-Xq_2(X,Y)$ is zero where by $p_2,q_2$ we denoted the second degree terms in $p,q$. We denote by $\eta$ this discriminant. The equation of this surface is
    $$\begin{aligned}
        \eta = & -32 - 27k^2 - 72km + 16m^2 + 32km^3 + 48n + \\
               & 36kmn - 16m^2n - 24n^2 + 4m^2n^2 + 4n^3=0. \\
    \end{aligned}$$

\medskip

\noindent \textbf{$C^{\infty}$ bifurcation surface in $\R^3$ due to a strong saddle or a strong focus changing the sign of their traces (weak saddle or weak focus)}

\medskip

\noindent {\bf (${\cal S}_{3}$)} This is the bifurcation surface due to the weakness of finite singularities, which occurs when the trace of a finite singular point is zero. The equation of this surface is given by $$\mathcal{T}_4=8+k^2+4n=0,$$ where $\mathcal{T}_4$ is defined in \cite{Vulpe:2011}. This $\mathcal{T}_4$ is an invariant.

\medskip

\noindent \textbf{$C^{\infty}$ bifurcation surface in $\R^3$ due to a node becoming a focus}

\medskip

\noindent {\bf (${\cal S}_{6}$)} This surface will contain the points of the parameter space where a finite node of the system turns into a focus. This surface is a $C^{\infty}$ but not a topological bifurcation surface. In fact, when we only cross the surface (${\cal S}_{6}$) in the bifurcation diagram, the topological phase portraits do not change. However, this surface is relevant for isolating the regions where a limit cycle surrounding an antisaddle (different from the triple node) cannot exist. Using the results of \cite{Artes-Llibre-Vulpe:2008}, the equation of this surface is given by $W_4=0$, where $$W_4=64 + 48k^2 + k^4 + 128km - 64n + 8k^2n + 16n^2.$$

These are all the bifurcation surfaces of singularities of the systems \eqref{eqtn} in the parameter space and they are all algebraic. We shall discover another bifurcation surface not necessarily algebraic and on which the systems have global connection of separatrices. The equation of this bifurcation surface can only be determined approximately by means of numerical tools. Using arguments of continuity in the phase portraits we can prove the existence of this not necessarily algebraic component in the region where it appears, and we can check it numerically. We will name it the surface (${\cal S}_{7}$).

\begin{figure}
\psfrag{m}{$m$} \psfrag{n}{$n$} \psfrag{k}{$k$}
\centerline{\psfig{figure=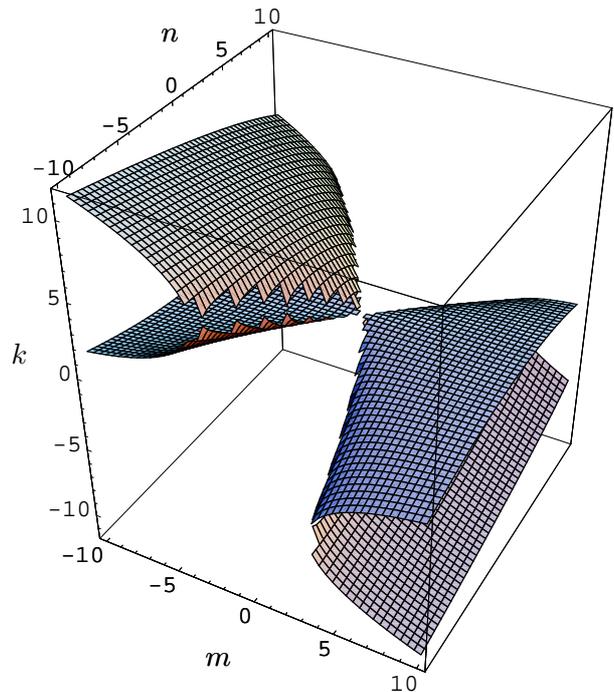,width=8cm}} \centerline {}
\caption{\small \label{fig:3d2} The 3--dimensional picture of the surface (${\cal S}_{6}$) (when a finite node becomes a focus).}
\end{figure}

\begin{remark} \rm
Even though we can draw a 3--dimensional picture of the algebraic bifurcation surfaces of singularities in $\R^3$ (see Fig. \ref{fig:3d2} for an example), it is pointless to try to see a single 3--dimensional image of all these four bifurcation surfaces together in the space $\R^3$. As we shall see later, the full partition of the parameter space obtained from all these bifurcation surfaces has 63 parts.
\end{remark}

Due to the above remark we shall foliate the 3--dimensional bifurcation diagram in $\R^3$ by planes $k=k_0$, $k_0$ constant. We shall give pictures of the resulting bifurcation diagram on these planar sections on an affine chart on $\R^2$. In order to detect the key values for this foliation, we must find the values of parameters where the surfaces intersect to each other. As we mentioned before, we will be only interested in non-negative values of $k$ to construct the bifurcation diagram.

The following set of seven results study the singularities of each surface and the simultaneous intersection points of the bifurcation surfaces, or the points or curves where two bifurcation surfaces are tangent.

As the final bifurcation diagram is quite complex, it is useful to introduce colors which will be used to talk about the bifurcation points:

    \begin{enumerate}[(a)]
        \item the curve obtained from the surface (${\cal S}_{1}$) is drawn in blue (a finite singular point collides with an infinite one);

        \item the curve obtained from the surface (${\cal S}_{3}$) is drawn in green (when the trace of a singular point becomes zero);

        \item the curve obtained from the surface (${\cal S}_{5}$) is drawn in red (two infinite singular points collide);

        \item the curve obtained from the surface (${\cal S}_{6}$) is drawn in black (an antisaddle different from the origin is on the verge of turning from a node to a focus or vice versa); and

        \item the curve obtained from the surface (${\cal S}_{7}$) is drawn in purple (the connection of separatrices).
    \end{enumerate}

\begin{lemma}\label{lem:lemma1} Concerning the singularities of the surfaces, it follows that:
    \begin{enumerate}[(i)]
        \item (${\cal S}_{1}$) and (${\cal S}_{3}$) have no singularities;
        \item (${\cal S}_{5}$) has a curve of singularities given by $4m^2+3n-6=0$;
        \item (${\cal S}_{6}$) has a singularity on the straight line $(m,2,0)$ on slice $k=0$. Besides, this surface restricted to $k=0$ is part of the surface (${\cal S}_{5}$).
    \end{enumerate}
\end{lemma}

\begin{proof} It is easy to see that the gradient of (${\cal S}_{1}$) and (${\cal S}_{3}$) is never null for all $(m,n,k) \in \R^3$; so \emph{(i)} is proved. In order to prove \emph{(ii)} we compute the gradient of $\eta$ and we verify that it is null whenever $m=-3\sqrt[3]{k}/2$ and $n=2-3\sqrt[3]{k^2}$, for all $k \geq 0$. It is easy to see that these values of $m$ and $n$ for all $k \geq 0$ lie on the curve $4m^2+3n-6=0$. Finally, considering the gradient of the surface (${\cal S}_{6}$), it is identically zero at the point $(0,2,0)$ which lies on the straight line $(m,2,0)$ whenever $k=0$. Moreover, if $k=0$, we see that the equation of (${\cal S}_{6}$) is $(-2+n)^2$, which is part of (${\cal S}_{5}$), proving \emph{(iii)}.
\end{proof}

\begin{lemma}\label{lem:lemma2} If $k=0$, the surfaces (${\cal S}_{1}$) and (${\cal S}_{3}$) have no intersection. For all $k>0$, they intersect in the point $(-(4+k^2)/2k,-2-k^2/4,k)$.
\end{lemma}

\begin{proof} By solving simultaneously both equations of the surfaces (${\cal S}_{1}$) and (${\cal S}_{3}$) for all $k>0$, we obtain the point $(-(4+k^2)/2k,-2-k^2/4,k)$. We also note that if $k=0$ there is no intersection point.
\end{proof}

\begin{lemma}\label{lem:lemma3} If $k=0$, the surfaces (${\cal S}_{1}$) and (${\cal S}_{5}$) intersect at the points $(-\sqrt{2},0,0)$ and $(\sqrt{2},0,0)$. For all $k>0$, they intersect along the curve $\gamma_1(m,n) = -64+32m^2+16n-n^2=0$ and they have a $2$--order contact along the curve $\gamma_2(m,n) = 1+2m^2+2n+m^2n+n^2=0$.
\end{lemma}

\begin{proof} By solving simultaneously both equations of the surfaces (${\cal S}_{1}$) and (${\cal S}_{5}$) for $k=0$, we obtain the two solutions $m_1=-\sqrt{2}, \ n_1=0$ and $m_2=\sqrt{2}, \ n_2=0$, proving the first part of the lemma. For all $k>0$, the simultaneous solutions of the equations are the three points: $r_1=(-\sqrt{2}-k/4,-\sqrt{2}k,k)$, $r_2=(\sqrt{2}-k/4,\sqrt{2}k,k)$ and $r_3=(-(4+k^2)/2k,-2-k^2/4,k)$. By computing the resultant with respect to $k$ of (${\cal S}_{1}$) and (${\cal S}_{5}$), we see that $\Res_k[({\cal S}_{1}),({\cal S}_{5})]=-16\gamma_1(m,n)(\gamma_2(m,n))^2$, where $\gamma_1(m,n)$ and $\gamma_2(m,n)$ are as stated in the statement of the lemma. It is easy to see that $\gamma_1(m,n)$ has two simple roots which are $r_1$ and $r_2$, and $r_3$ is a double root of $(\gamma_2(m,n))^2$. So that, the surfaces intersect transversally along the curve $\gamma_1(m,n)$ and they have a $2$--order contact along the curve $\gamma_2(m,n)$.
\end{proof}

\begin{lemma}\label{lem:lemma4} If $k=0$, the surfaces (${\cal S}_{1}$) and (${\cal S}_{6}$) have no intersection. For all $k>0$, they have a $2$--order contact along the curve $1+2m^2+2n+m^2n+n^2=0$.
\end{lemma}

\begin{proof} By solving the system formed by the equations of the surfaces (${\cal S}_{1}$) and (${\cal S}_{6}$), we find the point $r=(-(4+k^2)/2k,-2-k^2/4,k)$, for all $k>0$, which lies on the curve $1+2m^2+2n+m^2n+n^2=0$. We claim that the surfaces (${\cal S}_{1}$) and (${\cal S}_{6}$) have a 2--order contact point at $r$. Indeed, we have just shown that the point $r$ is a common point of both surfaces. Applying the change of coordinates given by $n=(v+km+k^2)/4$, $v \in \R$, we see that the gradient vector of (${\cal S}_{1}$) is $\nabla\mu(r)=(0,0,0)$ while the gradient vector of (${\cal S}_{6}$) is $\nabla W_{4}(r)=(0,0,8(-4+16/k^2+5k^2))$, whose last coordinate is always positive for all \linebreak $k>0$. As it does not change its sign, the vector $\nabla W_{4}(r)$ will always point upwards in relation to (${\cal S}_{1}$) restricted to the previous change of coordinates. Then, the surface (${\cal S}_{6}$) remains only on one of the half--spaces delimited by the surface (${\cal S}_{1}$), proving our claim.
\end{proof}

\begin{lemma}\label{lem:lemma5} If $k=0$, the surfaces (${\cal S}_{3}$) and (${\cal S}_{5}$) intersect at the points $(-2,-2,0)$ and $(2,-2,0)$. For all $k>0$, they intersect at the points $r_1=((32k-k^3-\sqrt{(64-k^2)^3})/256,-2-k^2/4,k)$, $r_2=(-(4+k^2)/2k,-2-k^2/4,k)$ and $r_3=((32k-k^3+\sqrt{(64-k^2)^3})/256,-2-k^2/4,k)$.
\end{lemma}

\begin{proof} The result follows easily by solving the system formed by the equations of the surfaces.
\end{proof}

\begin{corollary}\label{le:corollary-l5} If $k=2\sqrt{2}$, the points $r_1$ and $r_2$ of Lemma \ref{lem:lemma5} are equal and they correspond to the singularity of the surface (${\cal S}_{5}$).
\end{corollary}

\begin{proof} Replacing $k=2\sqrt{2}$ at the expressions of the points $r_1$, $r_2$ and $r_3$ described in Lemma \ref{lem:lemma5}, we see that $r_1=r_2$ and they are equal to the singularity $(-3\sqrt{2}/2,-4,2\sqrt{2})$ of the surface (${\cal S}_{5}$).
\end{proof}

\begin{remark} \label{remark:valuesofk} \rm
We observe that the values $k=0$ and $k=2\sqrt{2}$ will be very important to describe the bifurcation diagram due to the ``rich'' change on the behavior of the surfaces.
\end{remark}

\begin{lemma}\label{lem:lemma6} If $k=0$, the surfaces (${\cal S}_{5}$) and (${\cal S}_{6}$) intersect along the straight line $(m,2,0)$, for all \linebreak $m \in \R$. For all $k>0$, they have a $2$--order contact point at $(-(4+k^2)/2k,-2-k^2/4,k)$.
\end{lemma}

\begin{proof} Replacing $k=0$ in the equations of the surfaces and solving them in the variables $m$ and $n$, we find that $m\in\R$ and $n=2$, implying the existence of intersection along the straight line $(m,2,0)$, $m\in\R$. For all $k>0$, the solution of the equations of the surfaces is the point $r=(-(4+k^2)/2k,-2-k^2/4,k)$. We claim that the surfaces (${\cal S}_{5}$) and (${\cal S}_{6}$) have a 2--order contact point at $r$. We shall prove this claim by showing that each one of the surfaces (${\cal S}_{5}$) and (${\cal S}_{6}$) remains on only one of the half--spaces delimited by the plane (${\cal S}_{1}$) and their unique common point is $r$. Indeed, it is easy to see that the point $r$ is a common point of the three surfaces. By applying the change of coordinates given by $n=(v+km+k^2)/4$, $v \in \R$, as in the proof of Lemma \ref{lem:lemma4}, we see that the surface (${\cal S}_{6}$) remains on only one of the half--spaces delimited by the plane (${\cal S}_{1}$). On the other hand, numerical calculations show us that the surface (${\cal S}_{5}$) is zero valued around the point $r$ and it assumes negative values otherwise, showing that (${\cal S}_{5}$) remains on the other half--space delimited by the plane (${\cal S}_{1}$).
\end{proof}

\begin{lemma}\label{lem:lemma7}
The curve $r(k)=(-3\sqrt[3]{k}/2,\linebreak 2-3\sqrt[3]{k^2},k)$ of (${\cal S}_{5}$) (i.e., its set of singularities) cannot belong to the region where $W_{4} > 0$ \linebreak and $\mu < 0$.
\end{lemma}

\begin{proof}
We consider the real continuous function $g = (\mu W_{4}) \big|_{r(k)} = \left(\sqrt[3]{k^2}-2\right)^6 \left(\sqrt[3]{k^2}+6\right) \sqrt[3]{k^4}$, whose zeroes are $0$ and $2\sqrt{2}$. It is easy to see that $g$ is always positive in $(0,2\sqrt{2}) \cup (2\sqrt{2},\infty)$, implying that the functions $\mu$ and $W_{4}$ calculated at $r(k)$ cannot have different signs, proving the lemma.
\end{proof}

Now we shall study the bifurcation diagram having as reference the values of $k$ where significant phenomena occur in the behavior of the bifurcation surfaces.

According to the Remark \ref{remark:valuesofk}, these values are $k=0$ and $k=2\sqrt{2}$. So, we only need to add two more slices with some intermediate values.

We take then the values:
    \begin{equation}
        \aligned
            k_{0} & = 0, \ & k_{1} = 1,\\
            k_{2} & =2\sqrt{2}, \ & k_{3} = 3.\\
        \endaligned
    \end{equation}

The values indexed by positive even indices correspond to explicit values of $k$ for which there is a bifurcation in the behavior of the systems on the slices. Those indexed by odd ones are just intermediate points (see Figs. \ref{slicek0} to \ref{slicek3_no_purple}).

We now describe the labels used for each part. The subsets of dimensions 3, 2, 1 and 0, of the partition of the parameter space will be denoted respectively by $V$, $S$, $L$ and $P$ for Volume, Surface, Line and Point, respectively. The surfaces are named using a number which corresponds to each bifurcation surface which is placed on the left side of $S$. To describe the portion of the surface we place an index. The curves that are intersection of surfaces are named by using their corresponding numbers on the left side of $L$, separated by a point. To describe the segment of the curve we place an index. Volumes and Points are simply indexed (since three or more surfaces may be involved in such an intersection).

We consider an example: the surface (${\cal S}_{1}$) splits into 5 different two--dimensional parts labeled from $1S_{1}$ to $1S_{5}$, plus some one--dimensional arcs labeled as $1.iL_{j}$ (where $i$ denotes the other surface intersected by (${\cal S}_{1}$) and $j$ is a number), and some zero--dimensional parts. In order to simplify the labels in Figs. \ref{slice0_labels} to \ref{slice3_labels} we see \textbf{V1} which stands for the {\TeX} notation $V_{1}$. Analogously, \textbf{1S1} (respectively, \textbf{1.2L1}) stands for $1S_{1}$ (respectively, $1.2L_{1}$). And the same happens with many other pictures.

Some bifurcation surfaces intersect on $k=0$ or have singularities there. The restrictions of the surfaces on $k=0$ are: the surface (${\cal S}_{5}$) has a singularity at the point $(0,2,0)$ and it is the union of a parabola and a straight line of multiplicity two, which in turn coincides with the bifurcation surface (${\cal S}_{6}$); the surface (${\cal S}_{1}$) coincides with the horizontal axis and the bifurcation surface (${\cal S}_{3}$) becomes a straight line parallel to the horizontal line having intersection points only with the surface (${\cal S}_{5}$).

As an exact drawing of the curves produced by intersecting the surfaces with slices gives us very small regions which are difficult to distinguish, and points of tangency are almost impossible to recognize, we have produced topologically equivalent figures where regions are enlarged and tangencies are easy to observe. The reader may find the exact pictures in the web page http://mat.uab.es/$\sim$artes/articles/qvftn/qvftn.html.

As we increase the value of $k$, other changes in the bifurcation diagram happen. When $k=1$, the surface (${\cal S}_{5}$) has two connected components and a cusp point as a singularity which remains on the left side of the surface (${\cal S}_{6}$) until $k=2\sqrt{2}$ (see Fig. \ref{slicek1}). At this value, the cusp point is the point of contact among all the surfaces, as we can see in Fig. \ref{slicek2sqrt2}, and when $k=3$, the cusp point of (${\cal S}_{5}$) lies on the right side of surfaces (${\cal S}_{6}$) and (${\cal S}_{1}$) (see Fig. \ref{slicek3_no_purple} and Lemma \ref{lem:lemma7}). In order to comprehend that the ``movement'' of the cusp point of the surface (${\cal S}_{5}$) implies changes that occur in the bifurcation diagram, we see that when $k=1$ we have a ``curved triangular'' region formed by the surfaces (${\cal S}_{3}$) and (${\cal S}_{5}$), the cusp point of (${\cal S}_{5}$) and the points of intersection between both surfaces. The ``triangle'' bounded by these elements yields 15 subsets: three 3--dimensional subsets, seven 2--dimensional ones and five 1--dimensional ones. In Fig. \ref{slicek2sqrt2} we see that the ``triangle'' has disappeared and it has become a unique point which corresponds to the point of contact of all the surfaces and the cusp point of surface (${\cal S}_{5}$). Finally, when $k=3$, the ``triangle'' reappears and yields also 15 subsets of same dimensions, but different from the previous ones.

\begin{figure}
\centering
\psfrag{m}{$m$} \psfrag{n}{$n$}
\centerline{\psfig{figure=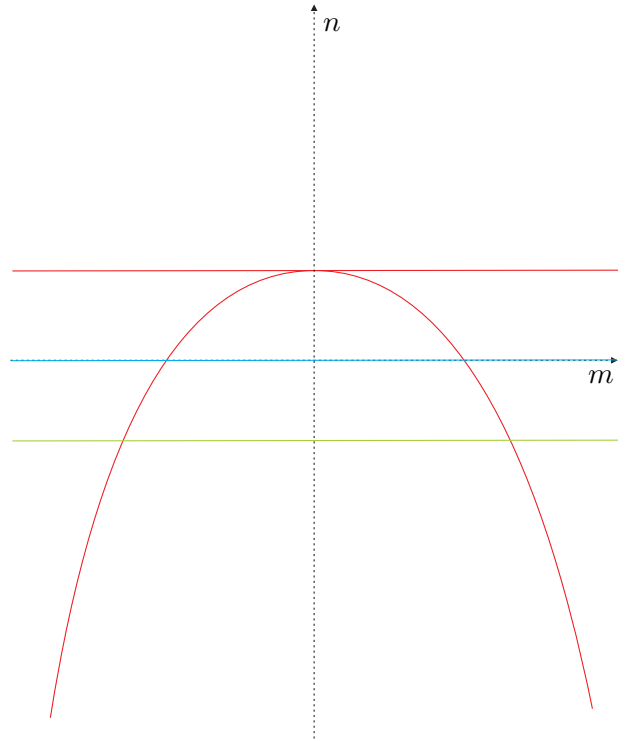,width=8cm}} \centerline {}
\caption{\small \label{slicek0} Slice of parameter space when $k=0$.}
\end{figure}

\begin{figure}
\centering
\psfrag{m}{$m$} \psfrag{n}{$n$}
\centerline{\psfig{figure=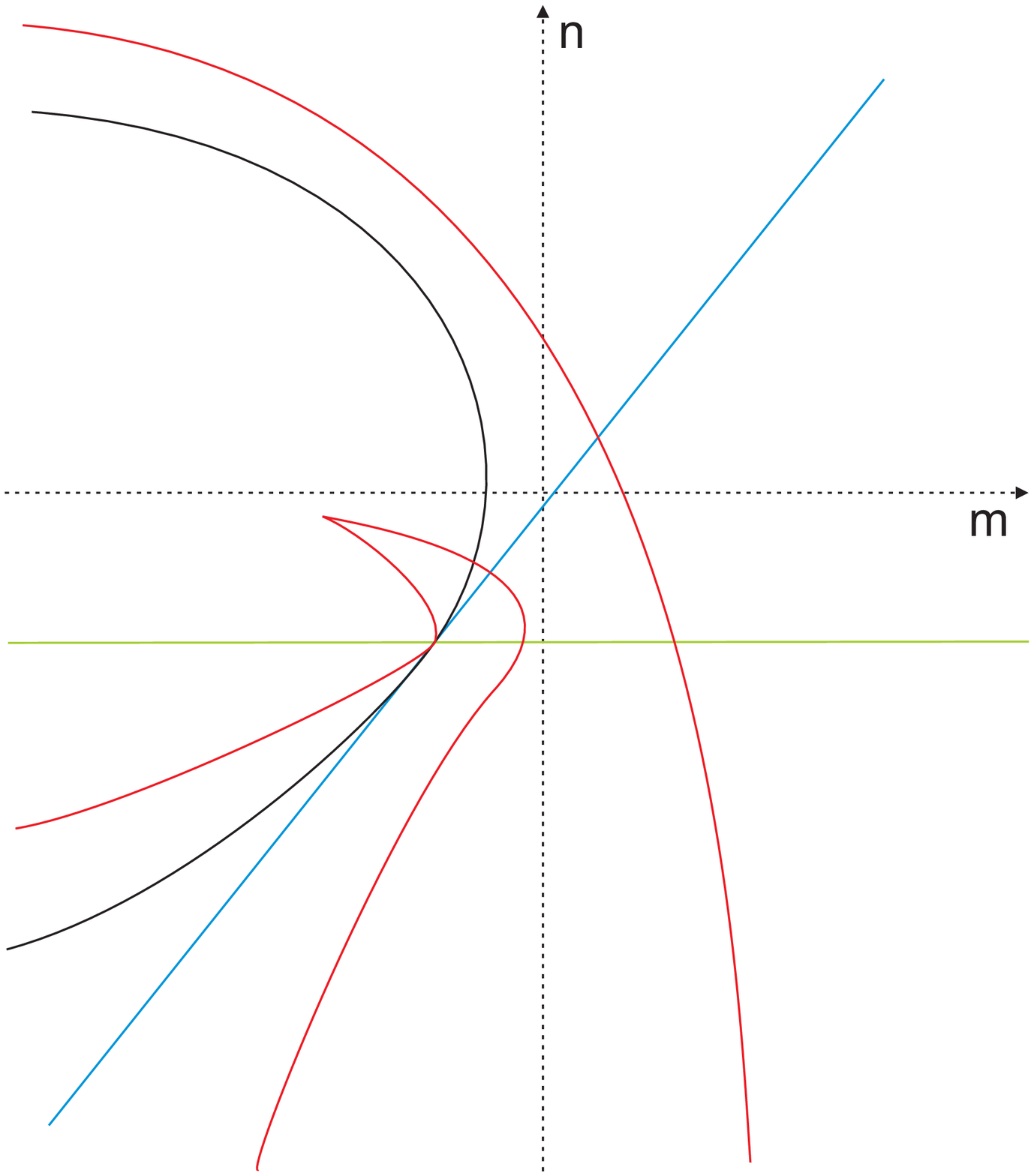,width=8cm}} \centerline {}
\caption{\small \label{slicek1} Slice of parameter space when $k=1$.}
\end{figure}

\begin{figure}
\centering
\psfrag{m}{$m$} \psfrag{n}{$n$}
\centerline{\psfig{figure=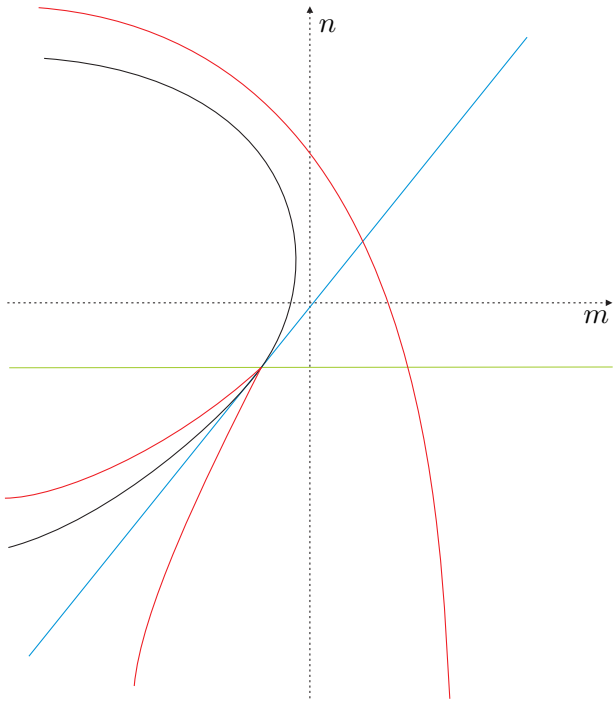,width=8cm}} \centerline {}
\caption{\small \label{slicek2sqrt2} Slice of parameter space when $k=2\sqrt{2}$.}
\end{figure}

\begin{figure}
\centering
\psfrag{m}{$m$} \psfrag{n}{$n$}
\psfrag{v5}{$v_{5}$}   \psfrag{3s1}{$3s_{1}$}   \psfrag{v6}{$v_{6}$}
\psfrag{5s4}{$5s_{4}$} \psfrag{5s6}{$5s_{6}$}   \psfrag{v14}{$v_{14}$}
\psfrag{v7}{$v_{7}$}   \psfrag{3s2}{$3s_{2}$}   \psfrag{6s2}{$6s_{2}$}
\psfrag{6s1}{$6s_{1}$}
\centerline{\psfig{figure=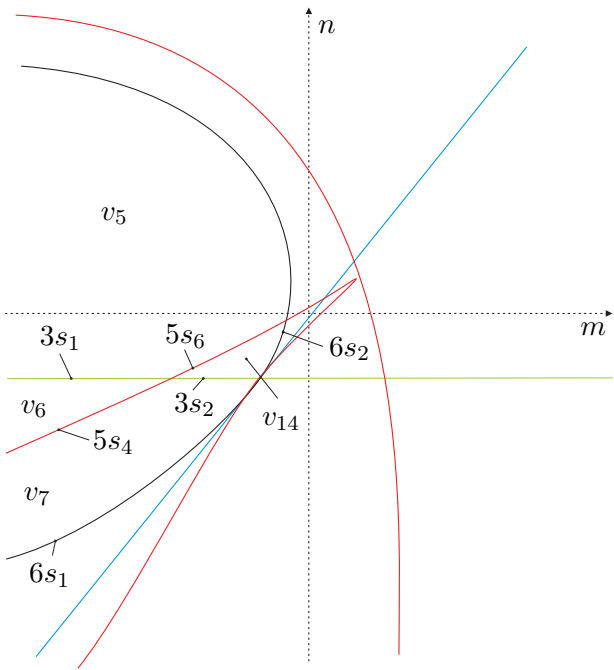,width=8cm}} \centerline {}
\caption{\small \label{slicek3_no_purple} Slice of parameter space when $k=3$.}
\end{figure}

All other regions of the parameter space related to singular points remain topologically the same with respect to the algebraic bifurcations of singularities when moving from Figs. \ref{slicek1} to \ref{slicek3_no_purple}.

We recall that the black curve (${\cal S}_{6}$) (or $W_{4}$) means the turning of a finite antisaddle different from the origin from a node to a focus. Then, according to the general results about quadratic systems, we could have limit cycles around such focus for any set of parameters having $W_{4} < 0$.

\begin{remark} \label{rem:f-n} \rm
Wherever two parts of equal dimension $d$ are separated only by a part of dimension $d-1$ of the black bifurcation surface (${\cal S}_{6}$), their respective phase portraits are topologically equivalent since the only difference between them is that a finite antisaddle has turned into a focus without change of stability and without appearance of limit cycles. We denote such parts with different labels, but we do not give specific phase portraits in pictures attached to Theorem \ref{th:1.1} for the parts with the focus. We only give portraits for the parts with nodes, except in the case of existence of a limit cycle or a graphic where the singular point inside them is portrayed as a focus. Neither do we give specific invariant description in Sec. \ref{sec:8} distinguishing between these nodes and foci.
\end{remark}

\subsection{Bifurcation surfaces due to connections}

We now place for each set of the partition on $k=3$ the local behavior of the flow around all the singular points. For a specific value of parameters of each one of the sets in this partition we compute the global phase portrait with the numerical program P4 \cite{Dumortier-Llibre-Artes:2006}. In fact, many (but not all) of the phase portraits in this work can be obtained not only numerically but also by means of perturbations of the systems of codimension one.

In this slice we have a partition in 2--dimensional regions bordered by curved polygons, some of them bounded, others bordered by infinity. From now on, we use lower letters provisionally to describe the sets found algebraically so not to interfere with the final partition described with capital letters. For each 2--dimensional region we obtain a phase portrait which is coherent with those of all their borders. Except one region. Consider the segment $3s_{1}$ in Fig. \ref{slicek3_no_purple}. On it we have a weak focus and a Hopf bifurcation. This means that either in $v_{5}$ or $v_{6}$ we must have a limit cycle. In fact it is in $v_{6}$. The same happens in $3s_{2}$, so a limit cycle must exist either in $v_{14}$ or $v_{7}$. However, when approaching $6s_{1}$ or $6s_{2}$, this limit cycle must have disappeared. So, either $v_{7}$ or $v_{14}$ must be splitted in two regions separated by a new surface (${\cal S}_{7}$) having at least one element $7S_{1}$ such that one region has limit cycle and the other does not, and the border $7S_{1}$ must correspond to a connection between separatrices. Numerically it can be checked that it is the region $v_{7}$ the one which splits in $V_{7}$ without limit cycles and $V_{15}$ with one limit cycle. It can also be analytically proved (see Proposition \ref{prop1}) that the segment $5s_{4}$ must be splitted in two segments $5S_{4}$ and $5S_{5}$ by the 1--dimensional subset $5.7L_{1}$. The other border of $7S_{1}$ must be $1.3L_{1}$ for coherence. We plot the complete bifurcation diagram in Fig. \ref{slice3_labels}. We also show the sequence of phase portraits along these subsets in Fig. \ref{sequence}.

Notice that the limit cycle which is ``born'' by Hopf on $3S_{1}$ either ``dies'' on $5S_{4}$ or ``survives'' when crossing $\eta = 0$, if we do it through $5S_{5}$, and then it ``dies'' either on $7S_{1}$ or again by Hopf in $3S_{2}$.

\begin{figure}
\centering
\psfrag{V5}{$V_{5}$}   \psfrag{3S1}{$3S_{1}$}     \psfrag{V6}{$V_{6}$}
\psfrag{5S4}{$5S_{4}$} \psfrag{5.7L1}{$5.7L_{1}$} \psfrag{5S5}{$5S_{5}$}
\psfrag{V15}{$V_{15}$} \psfrag{7S1}{$7S_{1}$}     \psfrag{V7}{$V_{7}$}
\psfrag{3S2}{$3S_{2}$} \psfrag{V14}{$V_{14}$}
\centerline{\psfig{figure=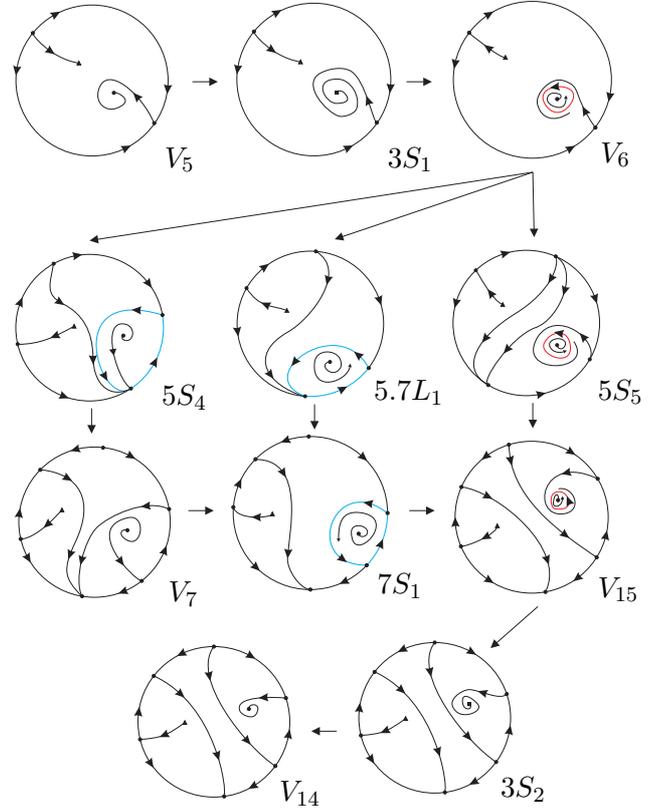,width=8cm}} \centerline {}
\caption{\small \label{sequence} Sequence of phase portraits in slice $k=3$. We start from $v_{5}$. We recall that the phase portrait $3S_{1}$ is equivalent to the phase portrait $V_{5}$ up to a weak focus (represented by a little black square) in place of the focus. When crossing $3s_{1}$, we shall obtain the phase portrait $V_{6}$ in subset $v_{6}$. From this point we may choose three different ways to reach the subset $v_{7}$ by crossing $5s_{4}$: (1) from the phase portrait $5S_{4}$ to the $V_{7}$; (2) from the phase portrait $5.7L_{1}$ to the $7S_{1}$; and (3) from the phase portrait $5S_{5}$ to the $V_{15}$, from where we can move to $V_{14}$.}
\end{figure}

Surface (${\cal S}_{7}$), for a concrete $k > 2\sqrt{2}$, is a curve which starts on $1.3L_{1}$ and may either cut $5s_{4}$, or not. We are going to prove that, at least for a concrete $k_0$, (${\cal S}_{7}$) must cut it, and consequently it must do the same for an open interval around $k_0$, thus proving the existence of subsets $5S_{4}$, $5S_{5}$ and $5.7L_{1}$ which have different phase portraits.


\begin{proposition} \label{prop1}
The following statements hold:
\begin{enumerate}[(a)]
    \item System \eqref{eqtn} with $(m,n,k) = (-29/2,-105/4,7)$ has an even number of limit cycles (counting their multiplicities), and possibly this number is zero;
    \item  System \eqref{eqtn} with $(m,n,k) = (-49/2,-185/4,12)$ has an odd number of limit cycles (counting their multiplicities), and possibly this number is one;
\end{enumerate}
\end{proposition}

\begin{proof}
\emph{(a)} We see that the system with rational coefficients
    \begin{equation}
        \begin{array}{ccl}
            \dot{x} & = & 2 xy + 7 y^2, \\
            \dot{y} & = & y - x^2 - 29 xy - \frac{105}{4} y^2 \\
        \end{array}
        \label{eq5S4}
    \end{equation}
is a representative of the red surface (${\cal S}_{5}$) which belongs to the subset $5s_{4}$.

We have to show that there exists a hyperbola $$\mathcal{H} \equiv ax^2 + bxy + y^2 + dx + ey + f = 0$$ which isolates the focus of \eqref{eq5S4} on the region where $\mathcal{H}<0$ and $x>0$, and with the property that at each of its points the flow crosses the hyperbola in only one direction, as we can see in Fig. \ref{hyp}. By proving the existence of this hyperbola, we shall prove that \eqref{eq5S4} has an even number of limit cycles.

\begin{figure}
\centering
\psfrag{5S4}{$5S_{4}$}
\centerline{\psfig{figure=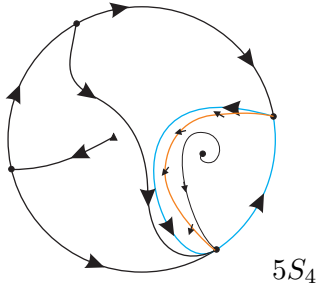,width=4cm}} \centerline {}
\caption{\small \label{hyp} The hyperbola in the phase portrait of $5S_{4}$.}
\end{figure}

For convenience and making easier the calculations, we impose that the hyperbola passes through two infinite singular points of \eqref{eq5S4} with the same tangencies of the affine separatrices. With all these features we have just one free parameter which is used to force the hyperbola to pass through a concrete finite point. In resume, we take $a = 1/14$, $b = 57/28$, $$d = \frac{1}{392} \left(\sqrt{148225 e^2+1}+399 e-1\right),$$ $$f = \frac{e^2 - 28d e + e \sqrt{(28d - e)^2}}{1516 \sqrt{(28d - e)^2} - 110}$$ and $e = -324/10000$.

This hyperbola has a component fully included in the fourth quadrant and it is easy to check that the scalar product of its tangent vector with the flow of the vector field does not change its sign and the flow moves outwards the region $\mathcal{H} < 0$. Since the focus is repellor, this is consistent with the absence of limit cycles (or with an even number of them, counting their multiplicities).

\emph{(b)} We see that the system with rational coefficients
    \begin{equation}
        \begin{array}{ccl}
            \dot{x} & = & 2 xy + 12 y^2, \\
            \dot{y} & = & y - x^2 - 49 xy - \frac{185}{4} y^2 \\
        \end{array}
        \label{eq5S5}
    \end{equation}
is a representative of the red surface (${\cal S}_{5}$) which belongs to the subset $5s_{4}$.

Analogously, we have to show that there exists a hyperbola $$\mathcal{H} \equiv ax^2 + bxy + y^2 + dx + ey + f = 0$$ which isolates the focus of \eqref{eq5S5} on the region where $\mathcal{H}<0$ and $x>0$, and with the property that at each of its points the flow crosses the hyperbola in only one direction, as we can see in Fig. \ref{hyp2}. By proving the existence of this hyperbola, we shall prove that \eqref{eq5S5} has an odd number of limit cycles.

\begin{figure}
\centering
\psfrag{5S5}{$5S_{5}$}
\centerline{\psfig{figure=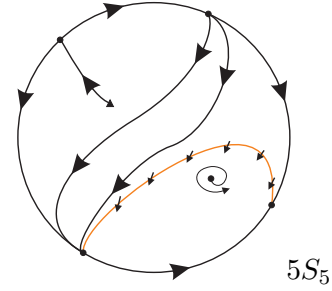,width=4cm}} \centerline {}
\caption{\small \label{hyp2} The hyperbola in the phase portrait of $5S_{5}$.}
\end{figure}

By using the same technique as before, we take $a = 1/24$, $b = 97/48$, $$d = \frac{1}{1152} \left(\sqrt{1299600 e^2 + 1} + 1164 e - 1\right),$$ $$f = \frac{e^2 - 48d e + e \sqrt{(48d - e)^2}}{4516 \sqrt{(48d - e)^2} - 190}$$ and $e = -18663/100000$.

This hyperbola has a component fully included in the fourth quadrant and it is easy to check that the scalar product of its tangent vector with the flow of the vector field does not change its sign and the flow moves inwards the region $\mathcal{H} < 0$. Since the focus is repellor, this is consistent with the presence of one limit cycle (or with an odd number of them, counting their multiplicities).
\end{proof}

We cannot be sure that this is all the additional bifurcation curves in this slice. There could exist others which are closed curves which are small enough to escape our numerical research. For all other two--dimensional parts of the partition of this slice whenever we join two points which are close to two different borders of the part, the two phase portraits are topologically equivalent. So we do not encounter more situations than the one mentioned above.

As we vary $k$ in $(2\sqrt{2}, \infty)$, the numerical research shows us the existence of the phenomenon just described, but for the values of $k$ in $[0,2\sqrt{2})$, we have not found the same behavior.

In Figs. \ref{slice0_labels} to \ref{slice3_labels} we show the complete bifurcation diagrams. In Sec. \ref{sec:8} the reader can look for the topological equivalences among the phase portraits appearing in the various parts and the selected notation for their representatives in Fig. \ref{fig:phase1}.

\begin{figure}
\centering
\psfrag{m}{$m$} \psfrag{n}{$n$}
\centerline{\psfig{figure=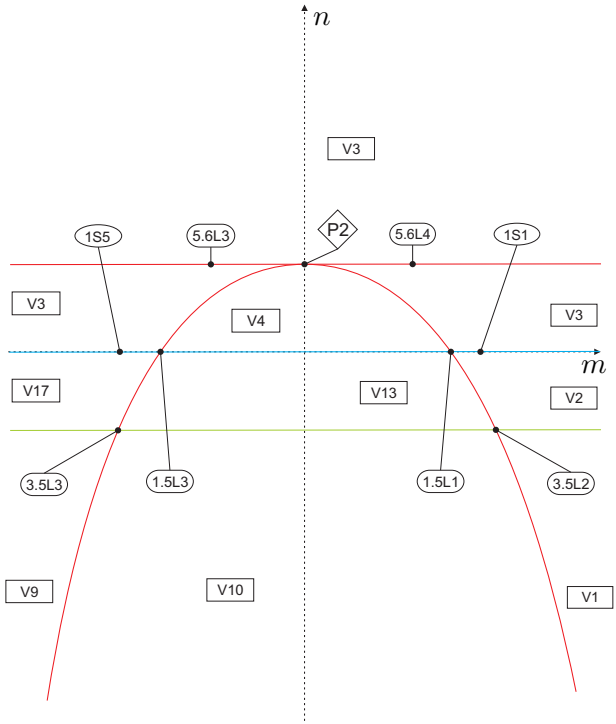,width=8cm}} \centerline {}
\caption{\small \label{slice0_labels} Complete bifurcation diagram for slice $k=0$.}
\end{figure}

\begin{figure}
\centering
\psfrag{m}{$m$} \psfrag{n}{$n$}
\centerline{\psfig{figure=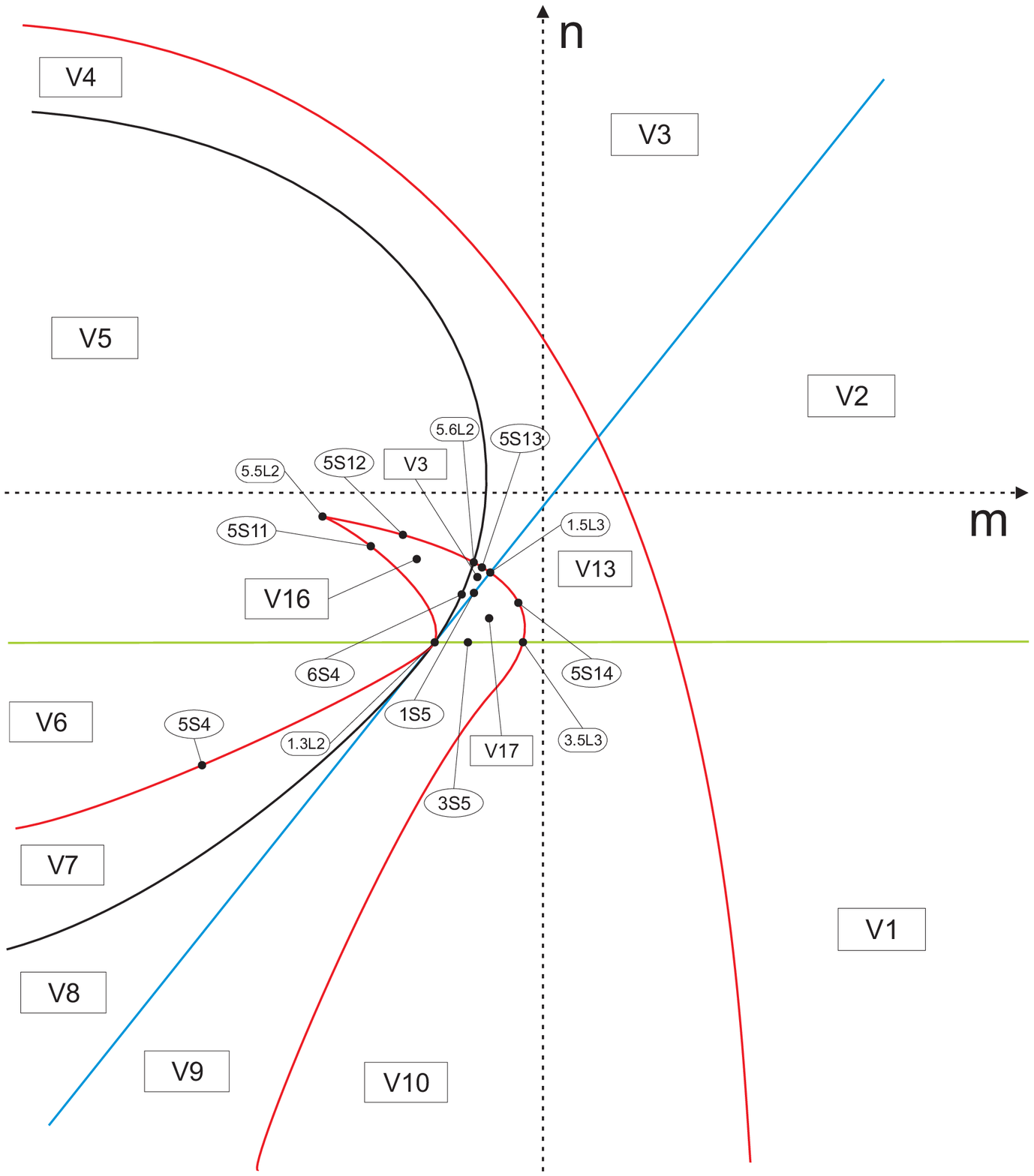,width=8cm}} \centerline {}
\caption{\small \label{slice1_labels} Complete bifurcation diagram for slice $k=1$.}
\end{figure}

\begin{figure}
\centering
\psfrag{m}{$m$} \psfrag{n}{$n$}
\centerline{\psfig{figure=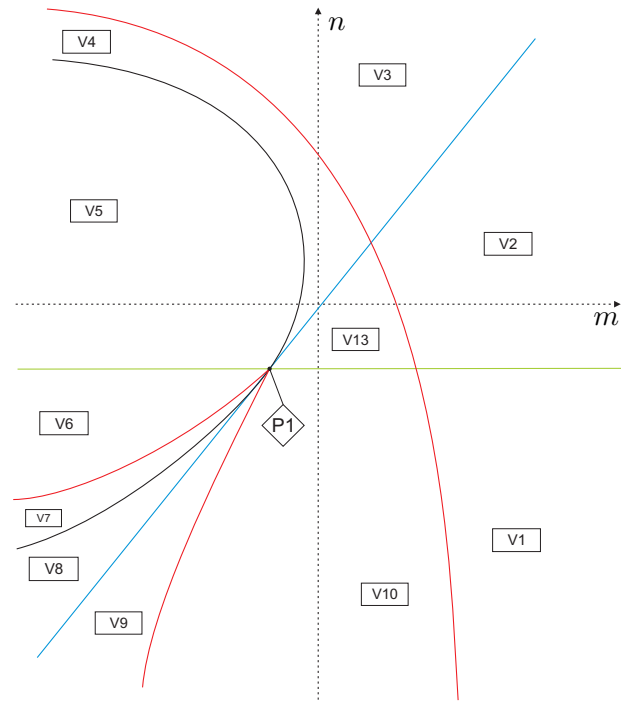,width=8cm}} \centerline {}
\caption{\small \label{slice2sqrt2_labels} Complete bifurcation diagram for slice $k=2\sqrt{2}$.}
\end{figure}

\begin{figure}
\centering
\psfrag{m}{$m$} \psfrag{n}{$n$}
\centerline{\psfig{figure=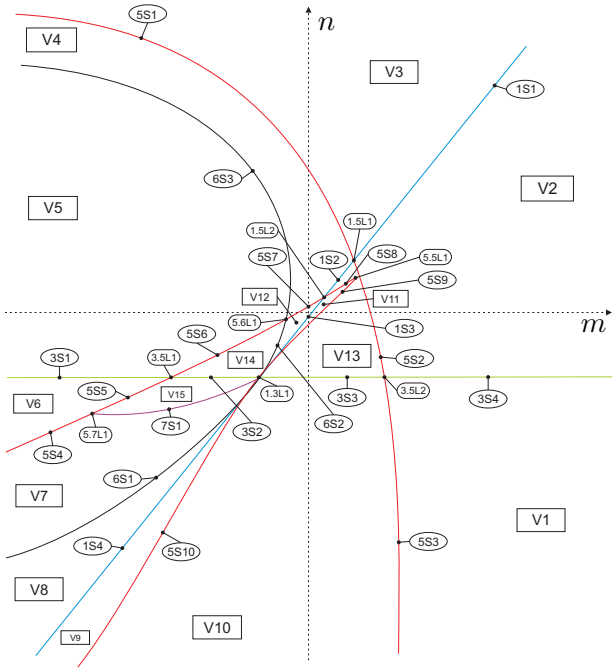,width=8cm}} \centerline {}
\caption{\small \label{slice3_labels} Complete bifurcation diagram for slice $k=3$.}
\end{figure}

\subsection{Other relevant facts about the bifurcation diagram} \label{sec:islands}

The bifurcation diagram we have obtained for $QT\overline{N}$ is completely coherent. By this, we mean that if we take any two points in the parameter space and join them by a continuous curve, along this curve the changes in phase portraits that occur when crossing the different bifurcation surfaces we mention can be completely explained.

However, we cannot be sure that this bifurcation diagram is the complete bifurcation diagram for $QT\overline{N}$ due to the possibility of ``islands'' inside the parts bordered by unmentioned bifurcation surfaces. In case they exist, these ``islands'' would not mean any modification of the nature of the singular points. So, on the border of these ``islands'' we could only have bifurcations due to saddle connections or multiple limit cycles.

In case there were more bifurcation surfaces, we should still be able to join two representatives of any two parts of the 63 parts found until now with a continuous curve either without crossing such bifurcation surface or, in case the curve crosses it, it must do it an even number of times without tangencies, otherwise one must take into account the multiplicity of the tangency, so the total number must be even. This is why we call these potential bifurcation surfaces ``\textit{islands}''.

To give an example of such a potential ``island'', we consider region $V_1$ where we have a phase portrait having a finite antisaddle, a saddle and two pairs of infinite antisaddles and one pair of infinite saddles. This phase portrait is topologically equivalent (modulo limit cycles and taking the triple node as a simple antisaddle) with the phase portrait 9.1 from \cite{Artes-Kooij-Llibre:1998} where all structurally stable quadratic vector fields were studied, (see the first phase portrait of Fig. \ref{fig:unstable}).

We note that in \cite{Artes-Kooij-Llibre:1998} it is proved that structurally stable (modulo limit cycles) quadratic vector fields can have exactly 44 different phase portraits. In the case of system \eqref{eqtn}, we have a semi--elemental triple node which topologically behaves like an elemental node, and the phase portraits in generic regions on bifurcation diagram will look like structurally stable ones. From those 44, two have no singular points, one has no finite antisaddles and 33 have four finite singular points, so obviously they cannot appear in $QT\overline{N}$. From the remaining 8, only 7 appear in our description of $QT\overline{N}$. There are two potential reasons for the absence of the remaining case: (1) it cannot be realized within $QT\overline{N}$, or (2) it may live in such ``islands'' where the conditions for the singular points are met, but the separatrix configuration is not the one that we have detected as needed for the coherence.

For example, the structurally stable phase portrait 9.2 has so far not appeared anywhere, but it could perfectly fit in an ``island'' inside $V_1$ (or $V_{11}$) where we have phase portrait 9.1. The transition from 9.1 to 9.2 consists on the existence of a heteroclinic connection between the finite saddle and one of the infinite saddles as it can be seen in Fig. \ref{fig:unstable}. We also show (in the middle of this figure) the unstable phase portrait from which could bifurcate and also has the potential to be on the bifurcation surface delimiting the ``island''.

\begin{figure}
\centering
\psfrag{9.1}{$9.1$} \psfrag{9.2}{$9.2$}
\centerline{\psfig{figure=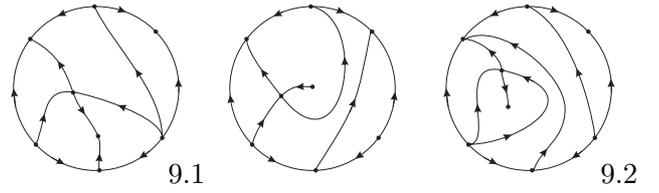,width=8cm}} \centerline {}
\caption{\small \label{fig:unstable} Example of a potential ``island''.}
\end{figure}

\section{Completion of the proof of the main theorem} \label{sec:8}

\indent In the bifurcation diagram we may have topologically equivalent phase portraits belonging to distinct parts of the parameter space. As here we have 63 distinct parts of the parameter space, to help us identify or to distinguish phase portraits, we need to introduce some invariants and we actually choose integer--valued invariants. All of them were already used in \cite{Llibre-Schlomiuk:2004,Artes-Llibre-Schlomiuk:2006}. These integer--valued invariants yield a classification which is easier to grasp.

\begin{definition} \rm We denote by $I_{1}(S)$ the number of the real finite singular points. This invariant is also denoted by $\N_{\R,f}(S)$ \cite{Artes-Llibre-Schlomiuk:2006}.
\end{definition}

\begin{definition}\rm We denote by $I_{2}(S)$ the sum of the indices of the real finite singular points. This invariant is also denoted by $\deg (DI_f(S))$ \cite{Artes-Llibre-Schlomiuk:2006}.
\end{definition}

\begin{definition}\rm We denote by $I_{3}(S)$ the number of the real infinite singular points. This invariant is also denoted by $\N_{\R,\infty}(S)$ \cite{Artes-Llibre-Schlomiuk:2006}.
\end{definition}

\begin{definition} \rm We denote by $I_{4}(S)$ the sequence of digits (each one ranging from 0 to 4) such that each digit describes the total number of global or local separatrices (different from the line of infinity) ending (or starting) at an infinite singular point. The number of digits in the sequences is 2, 4 or 6 according to the number of infinite singular points. We can start the sequence at anyone of the infinite singular points but all sequences must be listed in a same specific order either clockwise or counter--clockwise along the line of infinity.
\end{definition}

In our case we have used the clockwise sense beginning from the top--most infinite singular point in the pictures shown in Fig. \ref{fig:phase1}.

\begin{definition} \rm We denote by $I_{5}(S)$ a digit which gives the number of limit cycles.
\end{definition}

As we have noted previously in Remark \ref{rem:f-n}, we do not distinguish between phase portraits whose only difference is that in one we have a finite node and in the other a focus. Both phase portraits are topologically equivalent and they can only be distinguished within the $C^1$ class. In case we may want to distinguish between them, a new invariant may easily be introduced.

\begin{theorem} \label{t12}
Consider the family $QT\overline{N}$ of all quadratic systems with a semi--elemental triple node. Consider now all the phase portraits that we have obtained for this family. The values of the affine invariant ${\cal I}=(I_{1},I_{2},I_{3},I_{4},I_{5})$ given in the following diagram yield a partition of these phase portraits of the family $QT\overline{N}$.

Furthermore, for each value of $\cal I$ in this diagram there corresponds a single phase portrait; i.e. $S$ and $S'$ are such that $I(S)=I(S')$, if and only if $S$ and $S'$ are topologically equivalent.
\end{theorem}

The bifurcation diagram for $QT\overline{N}$ has 63 parts which produce 28 topologically different phase portraits as described in Table 1. The remaining 35 parts do not produce any new phase portrait which was not included in the 28 previous. The difference is basically the presence of a strong focus instead of a node and vice versa.

The phase portraits having neither limit cycle nor graphic have been denoted surrounded by parenthesis, for example $(V_{1})$; the phase portraits having one limit cycle have been denoted surrounded by brackets, for example $[V_{6}]$; the phase portraits having one graphic have been denoted surrounded by $\{\}$, for example $\{5S_{4}\}$.

\begin{proof}
The above result follows from the results in the previous sections and a careful analysis of the bifurcation diagrams given in Sec. \ref{sec:bifur} in Figs. \ref{slice0_labels} to \ref{slice3_labels}, the definition of the invariants $I_{j}$ and their explicit values for the corresponding phase portraits.
\end{proof}

We first make some observations regarding the equivalence relations used in this work: the affine and time rescaling, $C^1$ and topological equivalences.

The coarsest one among these three is the topological equivalence and the finest is the affine equivalence. We can have two systems which are topologically equivalent but not $C^1$--equivalent. For example, we could have a system with a finite antisaddle which is a structurally stable node and in another system with a focus, the two systems being topologically equivalent but belonging to distinct $C^1$--equivalence classes, separated by a surface (${\cal S}_{6}$ in this case) on which the node turns into a focus.

In Table 2 we listed in the first column 28 parts with all the distinct phase portraits of Fig. \ref{fig:phase1}. Corresponding to each part listed in column 1 we have in its horizontal block, all parts whose phase portraits are topologically equivalent to the phase portrait appearing in column 1 of the same horizontal block.

In the second column we have put all the parts whose systems yield topologically equivalent phase portraits to those in the first column but which may have some algebro--geometric features related to the position of the orbits.

In the third (respectively, fourth, and fifth) column we list all parts whose phase portraits have another antisaddle which is a focus (respectively, a node which is at a bifurcation point producing foci close to the node in perturbations, a node--focus to shorten, and a finite weak singular point).

Whenever phase portraits appear on a horizontal block in a specific column, the listing is done according to the decreasing dimension of the parts where they appear, always placing the lower dimensions on lower lines.

\onecolumn
\begin{center}
{\sc Table 1:} Geometric classification.
{
\[
I_{1}= \left\{ \ba{ll}
\mbox{$2$ \& } I_{2} = \left\{ \ba{l}
\mbox{$2$ \& $I_{3}=$} \left\{ \ba{l}

\mbox{$3$ \& $I_{4}=$} \left\{ \ba{l}
\mbox{$110110$} \,\, (V_{3}), \\
\mbox{$112110$} \,\, (V_{7}), \\
\mbox{$111111$ \& $I_{5}=$} \left\{ \ba{l}
\mbox{$1$} \,\, [V_{15}], \\
\mbox{$0$} \,\, (V_{12}), \ea \right. \\
\mbox{$110111$} \,\, \{7S_{1}\}, \ea \right. \\

\mbox{$2$ \& $I_{4}=$} \left\{ \ba{l}
\mbox{$1212$ \& $I_{5}=$} \left\{ \ba{l}
\mbox{$1$} \,\, [5S_{5}], \\
\mbox{$0$} \,\, (5S_{6}), \ea \right. \\
\mbox{$1111$} \,\, (5S_{1}), \\
\mbox{$1131$} \,\, \{5S_{4}\}, \\
\mbox{$1122$} \,\, (5S_{10}), \\
\mbox{$1121$} \,\, \{5.7L_{1}\}, \\ \ea \right. \\

\mbox{$1$ \& $I_{4}=$} \left\{ \ba{l}
\mbox{$11$ \& $I_{5}=$} \left\{ \ba{l}
\mbox{$1$} \,\, [V_{6}], \\
\mbox{$0$} \,\, (V_{4}), \ea \right. \\
\ea \right. \\
\ea \right.\\

\mbox{$0$ \& $I_{3}=$} \left\{ \ba{l}

\mbox{$3$ \& $I_{4}=$} \left\{ \ba{l}
\mbox{$111201$} \,\, (V_{1}), \\
\mbox{$110211$} \,\, (V_{9}), \\
\mbox{$101311$} \,\, (V_{11}), \\ \ea \right. \\

\mbox{$2$ \& $I_{4}=$} \left\{ \ba{l}
\mbox{$1122$} \,\, (5S_{2}), \\
\mbox{$2041$} \,\, (5S_{8}), \\
\mbox{$1132$} \,\, (5S_{9}), \\ \ea \right. \\

\mbox{$1$} \,\, (V_{10}), 

\ea \right.\\
\ea \right. \\

\mbox{$1$ \& $I_{2}=$} \left\{ \ba{l}
\mbox{$1$ \& $I_{3}=$} \left\{ \ba{l}

\mbox{$3$ \& $I_{4}=$} \left\{ \ba{l}
\mbox{$110110$} \,\, (1S_{1}), \\
\mbox{$102110$} \,\, (1S_{4}), \\
\mbox{$101111$} \,\, (1S_{3}), \\ \ea \right. \\

\mbox{$2$ \& $I_{4}=$} \left\{ \ba{l}
\mbox{$1200$} \,\, (1.3L_{1}), \\
\mbox{$1011$} \,\, \{1.3L_{2}\}, \\
\mbox{$1111$} \,\, (1.5L_{1}), \\
\mbox{$1202$} \,\, (1.5L_{2}), \\ \ea \right.\\

\mbox{$1$ \& $I_{4}=$} \left\{ \ba{l}
\mbox{$10$} \,\, (1S_{2}), \\
\mbox{$21$} \,\, (P_{1}). \\ \ea \right.\\

\ea \right.
    \ea \right.
        \ea \right.
\]
}
\end{center}
\twocolumn

\onecolumn
\begin{center}
{\sc Table 2:} Topological equivalences. \linebreak
\begin{tabular}{ccccc}
\hline
Presented & Identical       & Finite      & Finite       & Finite \\
phase     & under           & antisaddle  & antisaddle   & weak   \\
portrait  & perturbations   & focus       & node--focus  & point \\
\hline
\multirow{2}{*}{$V_{1}$} & $V_{2}$, $V_{9}$, $V_{17}$ & & & \\
        & & & & $3S_{4}$, $3S_{5}$ \\
\hline
\multirow{2}{*}{$V_{3}$} & & $V_{16}$ & & \\
        & & & $6S_{4}$ & \\
\hline
\multirow{3}{*}{$V_{4}$} & & $V_{5}$ & & \\
        & & & $6S_{3}$ & $3S_{1}$ \\
        & & $5.5L_{2}$ & & \\
        & $P_{2}$ & & & \\
\hline
$V_{6}$ & & & & \\
\hline
\multirow{2}{*}{$V_{8}$} & & $V_{7}$ & & \\
        & & & $6S_{1}$ & \\
\hline
\multirow{3}{*}{$V_{10}$} & $V_{13}$ & & & \\
         & & & & $3S_{3}$ \\
         & $5.5L_{1}$ & & & \\
\hline
$V_{11}$ & & & & \\
\hline
\multirow{2}{*}{$V_{12}$} & & $V_{14}$ & & \\
         & & & $6S_{2}$ & $3S_{2}$ \\
\hline
$V_{15}$ & & & & \\
\hline
$1S_{1}$ & $1S_{5}$ & & & \\
\hline
$1S_{2}$ & & & & \\
\hline
$1S_{3}$ & & & & \\
\hline
$1S_{4}$ & & & & \\
\hline
\multirow{2}{*}{$5S_{1}$} & $5S_{13}$ & $5S_{11}$, $5S_{12}$ & & \\
         & & & $5.6L_{2}$, $5.6L_{3}$, $5.6L_{4}$ & \\
\hline
\multirow{2}{*}{$5S_{2}$} & $5S_{3}$, $5S_{14}$ & & & \\
         & & & & $3.5L_{2}$, $3.5L_{3}$ \\
\hline
$5S_{4}$ & & & & \\
\hline
$5S_{5}$ & & & & \\
\hline
\multirow{2}{*}{$5S_{7}$} & & $5S_{6}$ & & \\
         & & & $5.6L_{1}$ & $3.5L_{1}$ \\
\hline
$5S_{8}$ & & & & \\
\hline
$5S_{9}$ & & & & \\
\hline
$5S_{10}$ & & & & \\
\hline
$7S_{1}$ & & & & \\
\hline
$1.3L_{1}$ & & & & \\
\hline
$1.3L_{2}$ & & & & \\
\hline
$1.5L_{1}$ & $1.5L_{3}$ & & & \\
\hline
$1.5L_{2}$ & & & & \\
\hline
$5.7L_{1}$ & & & & \\
\hline
$P_{1}$ & & & & \\
\hline
\end{tabular}
\end{center}
\twocolumn

\subsection {Proof of the main theorem}

The bifurcation diagram described in Sec. \ref{sec:bifur}, plus Table 1 of the geometrical invariants distinguishing the 28 phase portraits, plus Table 2 giving the equivalences with the remaining phase portraits lead to the proof of the main statement of Theorem \ref{th:1.1}.

In \cite{Artes-Llibre:2013} the authors are studying all phase portraits of quadratic systems having exactly one saddle--node or one connection of separatrices. By using a similar technique as the one used in \cite{Artes-Kooij-Llibre:1998} for the structurally stable ones, they have produced a complete list of topologically possible structurally unstable systems of codimension one (modulo limit cycles), they have erased many of them proving their impossibility and they have proved the existence of many others (180 just before this paper), and it remains 24 which escaped up to now both the proof of their impossibility and finding an example.

Our system in $QT\overline{N}$ $V_{11}$ yields an example of their ``wanted'' case $A_{23}$ by perturbating the triple node while producing the desired phase portrait as may be seen in Fig. \ref{sequence_unstable}.

\begin{figure}
\centering
\psfrag{A22}{$A_{22}$}   \psfrag{XX}{$V_{11}$}     \psfrag{A23}{$A_{23}$}
\centerline{\psfig{figure=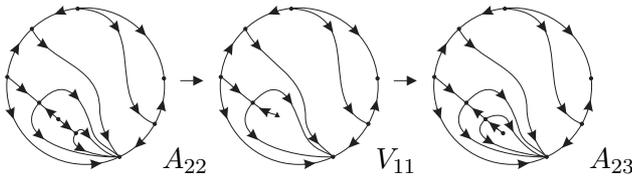,width=8cm}} \centerline {}
\caption{\small \label{sequence_unstable} The perturbations of phase portrait $V_{11}$ yielding the structurally unstable phase portraits $A_{22}$ and $A_{23}$.}
\end{figure}

\bigskip

\noindent\textbf {Acknowledgements.} The first author is partially supported by a MEC/FEDER grant number MTM 2008--03437 and by a CICYT grant number 2005SGR 00550, the second author is supported by CAPES/DGU grant number BEX 9439/12--9 and the last author is partially supported by CAPES/DGU grant number 222/2010 and by FP7--PEOPLE--2012--IRSES--316338.
%
%
%

\newcommand{\journal}[6]{#1 [#5] ``#2,'' \emph{#3} {\bf #4}, #6.}

\end{document}